\documentclass[a4paper, 10pt, parskip=half]{scrartcl}

\usepackage[utf8]{inputenc}
\usepackage[T1]{fontenc}
\usepackage{lmodern}
\usepackage{amsmath}
\usepackage{amssymb}
\usepackage{amsthm}
\usepackage{amsfonts}
\usepackage{dsfont}
\usepackage{mathtools}
\usepackage{marginnote}
\usepackage{xcolor}
\usepackage{hyperref}
\hypersetup{
  colorlinks=true,
  linkcolor=black,
  citecolor=black,
  urlcolor=blue,
  pdftitle={Approaching optimality in blow-up results for Keller--Segel systems with logistic-type dampening},
  pdfauthor={Mario Fuest},
  pdfkeywords={chemotaxis; finite-time blow-up; logistic source},
  bookmarksopen=true,
}
\newcommand{\tops}{\texorpdfstring}

\usepackage[numbers, sort&compress]{natbib}
\RequirePackage{geometry}
\newcommand{\R}{\mathbb{R}}

\newcommand{\ur}[1]{\mathrm{#1}}
\newcommand{\ure}{\ur e}

\ifdefined\labelenumi
  \renewcommand{\labelenumi}{(\roman{enumi})}
  
\fi

\newcommand{\eps}{\varepsilon}

\newcommand{\gt}{>}
\newcommand{\lt}{<}

\newcommand{\defs}{\coloneqq}
\newcommand{\sfed}{\eqqcolon}

\newcommand{\ra}{\rightarrow}

\newcommand{\nea}{\nearrow}

\newcommand{\sea}{\searrow}

\newcommand{\ol}{\overline}

\newcommand{\ds}{\,\mathrm{d}s}

\newcommand{\dsigma}{\,\mathrm{d}\sigma}

\newcommand{\drho}{\,\mathrm{d}\rho}

\newcommand{\ddt}{\frac{\mathrm{d}}{\mathrm{d}t}}

\newcommand{\hp}{\hphantom}
\newcommand{\pe}{\mathrel{\hp{=}}}

\newcommand{\tmax}{T_{\max}}

\newcommand{\intom}{\int_\Omega}

\newcommand{\Ombar}{\ol \Omega}

\newcommand{\leb}[1]{{L^{#1}(\Omega)}}

\newcommand{\con}[1]{{C^{#1}(\Ombar)}}

\let\originalparagraph\paragraph
\renewcommand{\paragraph}[2][.]{\originalparagraph{#2#1}}

\makeatletter
\renewenvironment{proof}[1][\proofname]{\par
  \pushQED{\qed}%
  \normalfont \topsep0\p@\relax
  \trivlist
  \item[\hskip\labelsep\scshape
  #1\@addpunct{.}]\ignorespaces
}{%
  \popQED\endtrivlist\@endpefalse
}
\makeatother

\newtheorem{base}{Base}[section]
\numberwithin{equation}{section}

\theoremstyle{plain}
\newtheorem{theorem}[base]{Theorem} \newtheorem*{theorem*}{Theroem}
\newtheorem{lemma}[base]{Lemma} \newtheorem*{lemma*}{Lemma}
 \newtheorem*{prop*}{Proposition}
 \newtheorem*{cor*}{Corollary}

\theoremstyle{definition}
 \newtheorem*{definition*}{Definition}
 \newtheorem*{example*}{Example}
 \newtheorem*{cond*}{Condition}

 \newtheorem*{remark*}{Remark}

\begin{document}
\setkomafont{title}{\normalfont \Large}
\title{Approaching optimality in blow-up results for Keller--Segel systems with logistic-type dampening}
\author{
Mario Fuest\footnote{fuestm@math.uni-paderborn.de}\\
{\small Institut f\"ur Mathematik, Universit\"at Paderborn,}\\
{\small 33098 Paderborn, Germany}
}
\date{}
\maketitle

\KOMAoptions{abstract=true}
\begin{abstract}
  \noindent
  Nonnegative solutions of the 
  Neumann initial-boundary value problem for the chemotaxis system
  \begin{align}\label{prob:star}\tag{$\star$}
    \begin{cases}
      u_t = \Delta u - \nabla \cdot (u \nabla v) + \lambda u - \mu u^\kappa, \\
      0 = \Delta v - \overline m(t) + u, \quad \overline m(t) = \frac1{|\Omega|} \int_\Omega u(\cdot, t)
    \end{cases}
  \end{align}
  in smooth bounded domains $\Omega \subset \mathbb R^n$, $n \ge 1$,
  are known to be global-in-time if $\lambda \geq 0$, $\mu > 0$ and $\kappa > 2$.\\[0.5pt]
  In the present work, we show that the exponent $\kappa = 2$ is actually critical in the four- and higher dimensional setting.
  More precisely, if
  \begin{alignat*}{3}
             \qquad n &\geq 4, &&\quad \kappa \in (1, 2) \quad &&\text{and} \quad \mu > 0 \\
    \text{or}\qquad n &\geq 5, &&\quad \kappa = 2        \quad &&\text{and} \quad \mu \in \left(0, \frac{n-4}{n}\right),
  \end{alignat*}
  for balls $\Omega \subset \mathbb R^n$ and parameters $\lambda \geq 0$, $m_0 > 0$,
  we construct a nonnegative initial datum $u_0 \in C^0(\overline \Omega)$ with $\int_\Omega u_0 = m_0$
  for which the corresponding solution $(u, v)$ of \eqref{prob:star} blows up in finite time.
  Moreover, in 3D, we obtain finite-time blow-up for $\kappa \in (1, \frac32)$ (and $\lambda \geq 0$, $\mu > 0$).\\[0.5pt]
  As the corner stone of our analysis,
  for certain initial data,
  we prove that the mass accumulation function $w(s, t) = \int_0^{\sqrt[n]{s}} \rho^{n-1} u(\rho, t) \,\mathrm d\rho$
  fulfills the estimate $w_s \le \frac{w}{s}$.
  Using this information, we then obtain finite-time blow-up of $u$
  by showing that for suitably chosen initial data, $s_0$ and $\gamma$,
  the function $\phi(t) = \int_0^{s_0} s^{-\gamma} (s_0 - s) w(s, t)$ cannot exist globally.\\[0.5pt]
  \textbf{Key words:} {chemotaxis; finite-time blow-up; logistic source}\\
  \textbf{MSC (2020):} {35B44 (primary); 35B33, 35K65, 92C17 (secondary)}
\end{abstract}
\section{Introduction}
A considerable amount of the literature on chemotaxis systems deals with detecting critical parameters distinguishing
between global existence and finite-time blow-up.
Such a dichotomy is already present in the \emph{minimal Keller--Segel system}
\begin{align}\label{prob:ks}
  \begin{cases}
    u_t = \Delta u - \nabla \cdot (u \nabla v), \\
    v_t = \Delta v - v + u
  \end{cases}  
\end{align}
proposed by Keller and Segel to model chemotactic behavior of bacteria attracted by a chemical substance they produce themselves~%
\cite{KellerSegelTravelingBandsChemotactic1971}.
Considered in two-dimensional balls, the mass of $u_0$ is critical:
If the initial datum $u_0$ is sufficiently regular, radially symmetric and satisfies $\intom u_0 \lt 8\pi$,
then the corresponding solutions are global-in-time and bounded~\cite{NagaiEtAlApplicationTrudingerMoserInequality1997}
while for any $m_0 \gt 8\pi$, there exists $u_0 \in \con0$ with $\intom u_0 = m_0$ leading to finite-time blow-up~%
\cite{HerreroVelazquezBlowupMechanismChemotaxis1997, MizoguchiWinklerBlowupTwodimensionalParabolic}. 
(See also \cite{NagaiBlowupRadiallySymmetric1995} for corresponding results in a parabolic--elliptic simplification of \eqref{prob:ks}.)
Let us note that this specific critical mass phenomenon is limited to the two-dimensional setting:
While solutions to \eqref{prob:ks} are always global-in-time an bounded 
if considered in one-dimensional domains \cite{OsakiYagiFiniteDimensionalAttractor2001},
in the spatially higher dimensional cases, finite-time blow-up has been detected
even for arbitrary positive initial masses \cite{WinklerFinitetimeBlowupHigherdimensional2013}.

Other dichotomies between boundedness and blow-up include critical exponents both
for nonlinear diffusion as well as nonlinear sensitivity~\cite{HorstmannWinklerBoundednessVsBlowup2005}
and nonlinear signal production~\cite{WinklerCriticalBlowupExponent2018}.
Instead of presenting them in detail here,
we refer to the surveys \cite{BellomoEtAlMathematicalTheoryKeller2015} and \cite{LankeitWinklerFacingLowRegularity2019}
for a broader overview of chemotaxis systems and related results.

Aiming to further enhance our understanding of the exact strength of the destabilising taxis term,
in this article, we present another critical parameter distinguishing between global existence and finite-time blow-up,
namely the exponent $\kappa = 2$ in Keller--Segel systems with logistic-type degradation.

Before stating our main result, let us introduce systems featuring such dampening terms and recall some of the corresponding results.
That is, we will first consider the \emph{Keller--Segel system with logistic source}
\begin{align}\label{prob:pe|pp_ls}
  \begin{cases}
    u_t = \Delta u - \nabla \cdot (u \nabla v) + \lambda u - \mu u^2, \\
    \tau v_t = \Delta v - v + u
  \end{cases}
\end{align}
in smooth, bounded domains $\Omega \subset \R^n$, $n \ge 2$, and given parameters $\lambda \in \R$, $\mu \gt 0$ and $\tau \ge 0$.
(We note that in view of the global existence result for $\lambda = \mu = 0$ in one-dimensional domains mentioned above,
at least for the question whether finite-time blow-up occurs, we may confine ourselves to the assumption $n \ge 2$.)
The system \eqref{prob:pe|pp_ls} and variations thereof describe several biological processes
such as
population dynamics~\cite{HillenPainterUserGuidePDE2009, ShigesadaEtAlSpatialSegregationInteracting1979},
pattern formation~\cite{WoodwardEtAlSpatiotemporalPatternsGenerated1995}
or embryogenesis~\cite{PainterEtAlDevelopmentApplicationsModel2000}
(see also \cite{HillenPainterUserGuidePDE2009} for an overview).

Already in 2007, Tello and Winkler showed that for $\tau = 0$, any $\lambda \in \R, \mu \gt \frac{n-2}{n}$ and any reasonably smooth initial data,
the system \eqref{prob:pe|pp_ls} possesses global, bounded classical solutions \cite{TelloWinklerChemotaxisSystemLogistic2007}.
Moreover, for $n \ge 3$ and $\mu = \frac{n-2}{n}$ (and again $\tau = 0$ and at least $\lambda \ge 0$)
solutions to \eqref{prob:pe|pp_ls} are global-in-time~\cite{KangStevensBlowupGlobalSolutions2016},
but to the best of our knowledge it is unknown whether these are also always bounded.
For the parabolic--parabolic case, that is, for $\tau \gt 0$, the situation is similar:
In the two-dimensional setting, assuming merely $\mu \gt 0$ suffices to guarantee global existence of classical solutions~\cite{OsakiEtAlExponentialAttractorChemotaxisgrowth2002}, even for dampening terms growing slightly slower then quadratically~\cite{XiangSublogisticSourceCan2018}.
Moreover, for higher dimensional convex domains, global classical solutions have been constructed for $\mu \gt \mu_0$ for some $\mu_0 \gt 0$ in \cite{WinklerBoundednessHigherdimensionalParabolic2010},
where explicit upper bounds of $\mu_0$ then have been derived in \cite{MuLinGlobalDynamicsFully2016, XiangHowStrongLogistic2018}
and the convexity assumption has been removed in \cite{XiangChemotacticAggregationLogistic2018} at the cost of worsening the condition on~$\mu$.
In all these settings, however, the known upper bounds for $\mu_0$ are larger than $\frac{n-2}{n}$.


However, if one resorts to more general solution concepts, further existence results are available.
Under rather mild conditions,
global weak solutions have been constructed
in \cite{TelloWinklerChemotaxisSystemLogistic2007} and \cite{LankeitEventualSmoothnessAsymptotics2015}
for the cases $\tau = 0$ and $\tau = 1$, respectively.
Moreover, if the degradation term $-\mu u^2$ in \eqref{prob:pe|pp_ls} is replaced by a weaker but still sufficiently strong superlinear dampening term,
global generalized solutions have been obtained,
again both for the parabolic--elliptic~\cite{WinklerChemotaxisLogisticSource2008} 
and the fully parabolic case~\cite{WinklerRoleSuperlinearDamping2019, WinklerSolutionsParabolicKeller2020, YanFuestWhenKellerSegel2020}.


On the other hand, it has been observed that despite the presence of quadratic dampening terms,
structures may form on intermediate time scales which even surpass so-called population thresholds to an arbitrary high extent
(cf.~\cite{KangStevensBlowupGlobalSolutions2016, LankeitChemotaxisCanPrevent2015, WinklerHowFarCan2014} for the parabolic--elliptic
and \cite{WinklerEmergenceLargePopulation2017} for the parabolic--parabolic case).

While these findings already show that the aggregating effect of the chemotaxis term
is strongly countered although not completely nullified by quadratic degradation terms,
the question arises whether the most drastic form of spatial aggregation---finite-time blow-up---%
still occurs in Keller--Segel systems with superlinear degradation terms.
A first partial (and affirmative) answer has been given in \cite{WinklerBlowupHigherdimensionalChemotaxis2011}:
There, the compared to \eqref{prob:pe|pp_ls} with $\tau = 0$ slightly simplified system
\begin{align}\label{prob:jl_ls}
  \begin{cases}
    u_t = \Delta u - \nabla \cdot (u \nabla v) + \lambda u - \mu u^\kappa, \\
    0 = \Delta v - \ol m(t) + u, \quad \ol m(t) \defs \frac{1}{|\Omega|} \intom u(\cdot, t)
  \end{cases}
\end{align}
is considered in balls $\Omega \subset \R^n$, $n \ge 5$
and, for any $\lambda \ge 0$ and $\kappa \in (1, \frac32 + \frac1{2(n-1)})$,
initial data leading to finite-time blow-up are constructed.
The second important finding in this direction transfers this result
to physically meaningful space dimensions.
More concretely, \cite{WinklerFinitetimeBlowupLowdimensional2018} detects finite-time blow-up even in the system~\eqref{prob:pe|pp_ls} with $\tau = 0$
(inter alia) for balls $\Omega \in \R^n$, $n \in \{3, 4\}$, $\lambda \ge 0$ and $\kappa \in (1, \frac76)$.

Recently, the regime of exponents allowing for finite-time blow-up in \eqref{prob:jl_ls} has been further widened
to $\kappa \gt \frac43$ and $\kappa \gt \frac32$ in the three- and four-dimensional settings, respectively \cite{BlackEtAlRelaxedParameterConditions2020}.
Moreover, in planar domains, chemotactic collapse can be obtained if one replaces the term $-u^\kappa$ in \eqref{prob:jl_ls}
with certain heterogeneous dampening terms such as $-|x|^2 u^2$~%
\cite{FuestFinitetimeBlowupTwodimensional2020}.
(Let us additionally note that similar finite-time blow-up results are also available for systems with nonlinear diffusion~%
\cite{BlackEtAlRelaxedParameterConditions2020, LinEtAlBlowupResultQuasilinear2018}
or sublinear taxis sensitivity \cite{TanakaYokotaBlowParabolicElliptic2020}.)

\paragraph{Main results}
At least in the four- and higher dimensional settings,
the journey of detecting finite-time blow-up in \eqref{prob:jl_ls} for ever increasing values of $\kappa$
comes to an end with the present article;
we obtain the corresponding result up to (and for $n \ge 5$ even including) the optimal exponent $\kappa = 2$.
 
More precisely, our main result reads
\begin{theorem}\label{th:main}
  Suppose
  \begin{subequations}\label{eq:main:cond}
  \begin{alignat}{3}
                    n &\ge 3, &&\quad \kappa \in \left(1, \min\left\{2, \frac n2\right\}\right) \quad &&\text{and} \quad \mu \gt 0 \label{eq:main:cond:1} \\
    \text{or}\qquad n &\ge 5, &&\quad \kappa = 2        \quad &&\text{and} \quad \mu \in \left(0, \frac{n-4}{n}\right). \label{eq:main:cond:2}
  \end{alignat}
  \end{subequations}
  Moreover, let $R \gt 0$, $\Omega \defs B_R(0) \subset \R^n$, $m_0 \gt 0$, $m_1 \in (0, m_0)$ and $\lambda \ge 0$.
  Then there exists $r_1 \in (0, R)$ 
  such that whenever
  \begin{align}\label{eq:main:reg_u0}
    u_0 \in \con1 \quad \text{is positive, radially symmetric as well as radially decreasing}
  \end{align}
  and fulfills
  \begin{align}\label{eq:main:mass_u0}
    \intom u_0 = m_0
    \quad \text{as well as} \quad
    \int_{B_{r_1}(0)} u_0 \ge m_1,
  \end{align}
  the following holds:
  There exist $\tmax \lt \infty$ and a classical solution
  \begin{align}\label{eq:main:uv_reg}
    (u, v) \in \left( C^0(\Ombar \times [0, \tmax)) \cap C^{2, 1}(\Ombar \times (0, \tmax)) \right)^2
  \end{align}
  of
  \begin{align}\label{prob}
    \begin{cases}
      u_t = \Delta u - \nabla \cdot (u \nabla v) + \lambda u - \mu u^\kappa                 & \text{in $\Omega \times (0, \tmax)$}, \\
      0 = \Delta v - \ol m(t) + u, \quad \ol m(t) \defs \frac1{|\Omega|} \intom u(\cdot, t) & \text{in $\Omega \times (0, \tmax)$}, \\
      \partial_\nu u = \partial_\nu v = 0                                                   & \text{on $\partial \Omega \times (0, \tmax)$}, \\
      u(\cdot, 0) = u_0                                                                     & \text{in $\Omega$},
    \end{cases}
  \end{align}
  which blows up at $\tmax$ in the sense that $\lim_{t \nea \tmax} u(0, t) = \infty$.
\end{theorem}

\paragraph{Main ideas}
Following Jäger and Luckhaus \cite{JagerLuckhausExplosionsSolutionsSystem1992},
we rely on the mass accumulation function given by $w(s, t) \defs \int_0^{\sqrt[n]s} \rho^{n-1} u(\rho, t) \drho$,
which transforms \eqref{prob} to a scalar equation, see Lemma~\ref{lm:def_w}.
The predecessors \cite{BlackEtAlRelaxedParameterConditions2020} and \cite{WinklerFinitetimeBlowupLowdimensional2018} of this article,
which deal with (variations of) the system \eqref{prob:jl_ls},
then proceed to show that the function $\phi$ defined by
\begin{align}\label{eq:intro:phi}
  \phi(s_0, t) \defs \int_0^{s_0} s^{-\gamma} (s_0 - s) w(s, t) \ds 
\end{align}
cannot, at least not for certain initial data, $s_0 \in (0, R^n)$ and $\gamma \in (0, 1)$,
exist globally in time, implying that $u$ must blow up in finite time.
One of the most challenging terms to estimate arises from the degradation term; one essentially has to control the integral
$\int_0^{s_0} w_s^{\kappa}(s, t) \ds$.
At this point, pointwise estimates for $w_s$ come in handy,
which due to the identity $w_s(s, t) = u(s^\frac1n, t)$ are available once pointwise estimates for $u$ are known.
These in turn can for instance be obtained by analyzing general parabolic equations in divergence form~%
\cite{FuestBlowupProfilesQuasilinear2020, WinklerBlowupProfilesLife} or by arguments similar to Lemma~\ref{lm:ws_est} below.
In fact, one of the main points in \cite{BlackEtAlRelaxedParameterConditions2020}
is to discuss how pointwise upper estimates for $u$ of the form $u(x) \le C |x|^{-p}$ influence the possibility to detect finite-time blow-up.

However, a natural limitation of this approach is the exponent $p = n$ since for fixed $C \gt 0$ and $p \lt n$,
nonnegative functions $u_0 \in \con0$ with $u_0(x) \le C |x|^{-p}$ cannot have their mass concentrated arbitrary close to the origin;
that is, depending on the value of $C$ and $p$, none of these functions may fulfill \eqref{eq:main:mass_u0}.
However, as seen in \cite{BlackEtAlRelaxedParameterConditions2020},
even the choice $p=n$ `only' yields finite-time blow-up for the system \eqref{prob:jl_ls} for certain $\kappa \lt \frac32$.

Thus, in the present article, where we handle exponents $\kappa \le 2$, we choose a slightly different path.
At the basis of our analysis stands Lemma~\ref{lm:ws_est}:
There, we derive the key estimate $w_s \le \frac{w}{s}$, 
which due to $w(0, \cdot) \equiv 0$ actually improves on $w_s \le \frac{C}{s}$.
Its proof is surprisingly simple: As already observed in similar contexts
(cf.~\cite{BlackEtAlRelaxedParameterConditions2020, FuestFinitetimeBlowupTwodimensional2020, WinklerCriticalBlowupExponent2018}),
for radially decreasing initial data,
$w_s(\cdot, t)$ is decreasing for all times $t$, see Lemma~\ref{lm:ur_le_0}.
The desired estimate is then just a consequence of the mean value theorem.

Another major difference of our methods compared to
\cite{BlackEtAlRelaxedParameterConditions2020} and \cite{WinklerFinitetimeBlowupLowdimensional2018}
is that we do not limit our analysis of \eqref{eq:intro:phi} to $\gamma \in (0, 1)$ but also allow for parameters $\gamma$ being larger than $1$.
In the five- and higher dimensional settings, this will then allow us to obtain finite-time blow-up even for $\kappa = 2$.
(In 3D and 4D, the term stemming from the diffusion forces $\gamma$ to be smaller than $1$
and hence we cannot employ the same method as in higher dimensions.)
We also note that the realization of the idea of taking $\gamma \gt 1$
is made possible by the new crucial estimate $w_s \le \frac{w}{s}$.

The rest of the article is organized as follows:
After stating some preliminary results in Section~\ref{sec:prelim},
in Section~\ref{sec:w} we derive $w_s \le \frac{w}{s}$ in Lemma~\ref{lm:ws_est}.
Section~\ref{sec:phi} then starts with the definition of the function $\phi$ and a calculation of its derivative, see Lemma~\ref{lm:phi},
Next, in the Lemma~\ref{lm:i4}, we suitably estimate the term originating in the logistic source,
before dealing with the remaining terms and the initial datum of $\phi$ in the subsequent lemmata.
In Lemma~\ref{lm:tmax_le_12}, we then finally prove finiteness of the maximal existence time $\tmax$.

\section{Preliminaries}\label{sec:prelim}
In the sequel, we fix $n \ge 3$, $R \gt 0$, $\Omega \defs B_R(0) \subset \R^n$, $\kappa \in (1, 2]$, $\lambda \ge 0$ and $\mu \gt 0$.
\begin{lemma}\label{lm:local_ex}
  Suppose $u_0$ complies with $\eqref{eq:main:reg_u0}$.
  There exists $\tmax \in (0, \infty]$ and a unique pair $(u, v)$ of regularity \eqref{eq:main:uv_reg} which solves \eqref{prob} classically
  and is such that if $\tmax \lt \infty$, then $\lim_{t \nea \tmax} \|u(\cdot, t)\|_{\leb\infty} = \infty$.
  Moreover, both $u$ and $v$ are radially symmetric and $u$ is positive in $\Ombar \times [0, \tmax)$.
\end{lemma}
\begin{proof}
  This is contained in \cite[Lemma~1.1]{WinklerBlowupHigherdimensionalChemotaxis2011}. 
\end{proof}

Given $u_0$ as in \eqref{eq:main:reg_u0}, we denote the solution given in Lemma~\ref{lm:local_ex} by $(u, v)$ and its maximal existence time by $\tmax$.
Moreover, we always set $\ol m(t) \defs \frac1{|\Omega|} \intom u(\cdot, t)$ for $t \in [0, \tmax)$.

Since the zeroth order term in the first equation in \eqref{prob:ks}, $\lambda u - \mu u^\kappa$, grows at most linearly in $u$,
we can easily control the mass of the first solution component.
\begin{lemma}\label{lm:u_bdd_l1}
  Suppose that $u_0$ satisfies \eqref{eq:main:reg_u0}.
  Then
  \begin{align*}
    \intom u(\cdot, t) \le \ure^{\lambda t} \intom u_0
    \qquad \text{for all $t \in (0, \tmax)$}.
  \end{align*}
\end{lemma}
\begin{proof}
  This immediately follows from integrating the first equation in \eqref{prob} and using that $\mu \gt 0$.
\end{proof}

As used multiple times in the sequel, let us also state the following elementary
\begin{lemma}\label{lm:beta}
  Given $a \gt -1$, there is $B \in (0, \infty)$ such that for any $s_0 \gt 0$, the identity
  \begin{align*}
    \int_0^{s_0} s^a (s_0 - s) \ds = B s_0^{a+2}
  \end{align*}
  holds.
\end{lemma}
\begin{proof}
  We substitute $s \mapsto s_0 s$ and take $B \defs \int_0^1 s^a (1-s) \ds \in (0, \infty)$.
\end{proof}

\section{The mass accumulation function \tops{$w$}{w}}\label{sec:w}
Given $u_0$ as in \eqref{eq:main:reg_u0} (and thus $(u, v)$ as in Lemma~\ref{lm:local_ex}),
we denote the mass accumulation function by
\begin{align}\label{eq:def_w:def}
  w(s, t) \defs \int_0^{s^\frac1n} \rho^{n-1} u(\rho, t) \drho, \qquad (s, t) \in [0, R^n] \times [0, \tmax),
\end{align}
which has been introduced in the context of chemotaxis systems in \cite{JagerLuckhausExplosionsSolutionsSystem1992}.
In this section, we prove some of its properties, most notably the crucial estimate $w_s \le \frac{w}{s}$ in Lemma~\ref{lm:ws_est}.

\begin{lemma}\label{lm:def_w}
  For every $u_0$ satisfying \eqref{eq:main:reg_u0},
  the function $w$ given by \eqref{eq:def_w:def}
  belongs to $C^0([0, R^n] \times [0, \tmax)) \cap C^{2, 1}([0, R^n] \times (0, \tmax))$ and fulfills
  \begin{align}\label{eq:def_w:w_s=u}
    w_s(s, t) = \frac{u(s^\frac1n, t)}{n}
    \qquad \text{for all $(s, t) \in [0, R^n] \times [0, \tmax)$}
  \end{align}
  as well as
  \begin{align}\label{eq:def_w:w_pde_exact}
          w_t
    &=    n^2 s^{2-\frac2n} w_{ss}
          + n w w_s
          - n \ol m(t) s w_s
          + \lambda w
          - n^{\kappa-1} \mu \int_0^s w_s^{\kappa}(\sigma, t) \dsigma
    \qquad \text{in $(0, R^n) \times (0, \tmax)$}.
  \end{align}
\end{lemma}
\begin{proof}
  This can be seen by a direct calculation.
  In fact,
  the asserted regularity is a consequence of Lemma~\ref{lm:local_ex},
  the identity \eqref{eq:def_w:w_s=u} follows from the chain rule,
  and \cite[equation~(1.4)]{WinklerBlowupHigherdimensionalChemotaxis2011} asserts that \eqref{eq:def_w:w_pde_exact} holds.
\end{proof}

Next, as a major step towards proving $w_s \le \frac{w}{s}$,
we show that for initial data fulfilling \eqref{eq:main:reg_u0}, the first solution component is radially decreasing throughout evolution.
\begin{lemma}\label{lm:ur_le_0}
  Suppose $u_0$ complies with \eqref{eq:main:reg_u0}.
  Then $u_r \le 0$ in $(0, R) \times (0, \tmax)$.
\end{lemma}
\begin{proof}
  This can be shown as in \cite[Lemma~5.1]{BlackEtAlRelaxedParameterConditions2020} or \cite[Lemma~3.7]{FuestFinitetimeBlowupTwodimensional2020}
  (which in turn both follow \cite[Lemma~2.2]{WinklerCriticalBlowupExponent2018}).
  However, due to the importance of this lemma for showing the crucial estimate $w_s \le \frac{w}{s}$ in the succeeding lemma,
  we choose to at least sketch the proof here.
  First, by an approximation argument as in \cite[Lemma~2.2]{WinklerCriticalBlowupExponent2018},
  we may assume $u_r \in C^0([0, R] \times [0, \tmax)) \cap C^{2, 1}((0, R) \times (0, \tmax))$.
  
  Furthermore, the second equation in \eqref{prob} asserts
  \begin{align*}
      r^{1-n} (r^{n-1} u v_r)_r
    = u_r v_r + u r^{1-n} (r^{n-1} v_r)_r
    = u_r v_r - u^2 + \ol m(t) u
    \qquad \text{in $(0, R) \times (0, \tmax)$}
  \end{align*}
  and hence
  \begin{align*}
          u_{rt}
    &=    \left( r^{1-n} \left( r^{n-1} (u_r - u v_r) \right)_r + f(u) \right)_r \\
    &=    u_{rrr} + \frac{n-1}{r} u_{rr} - \frac{n-1}{r^2} u_r - u_{rr} v_r - u_r v_{rr} + 2 u u_r - \ol m(t) u_r + f'(u) u_r \\
    &=    u_{rrr} + a(r, t) u_{rr} + b(r, t) u_r
    \qquad \text{in $(0, R) \times (0, \tmax)$},
  \end{align*}
  where
  \begin{align*}
    a(r, t) \defs \frac{n-1}{r} - v_r(r, t)
    \quad \text{and} \quad
    b(r, t) \defs - \frac{n-1}{r^2} - v_{rr}(r, t) + 2 u(r, t) - \ol m(t) + f'(u(r, t))
  \end{align*}
  for $(r, t) \in (0, R) \times (0, \tmax)$.

  As can be rapidly seen by writing the second equation in \eqref{prob} in radial coordinates
  (and has been argued in more detail in \cite[Lemma~3.6]{FuestFinitetimeBlowupTwodimensional2020}, for instance),
  $-v_{rr} \le u$ holds throughout $(0, R) \times (0, \tmax)$, so that for fixed $T \in (0, \tmax)$, we can estimate
  \begin{align*}
        \sup_{r \in (0, R), t \in (0, T)} b(r, t)
    \le 3 \|u\|_{L^\infty((0, R) \times (0, T))} + \|f'\|_{L^\infty(0, \|u\|_{L^\infty((0, R) \times (0, T))})}
    \lt \infty.
  \end{align*}
  An application of the maximum principle (cf.\ \cite[Proposition~52.4]{QuittnerSoupletSuperlinearParabolicProblems2007})
  then gives $u_r \le 0$ in $(0, R) \times (0, T)$, which upon taking $T \nea \tmax$ implies the statement.
\end{proof}

As already advertised multiple times, this lemma now allows us to rapidly obtain the important estimate $w_s \le \frac{w}{s}$.
\begin{lemma}\label{lm:ws_est}
  Assume that $u_0$ satisfies \eqref{eq:main:reg_u0}.
  For all $s \in [0, R^n]$ and $t \in [0, \tmax)$,
  \begin{align}\label{eq:ws_est:main}
    w_s(s, t) \le \frac{w(s, t)}{s} \le w_s(0, t)
  \end{align}
  holds.
  In particular, for all $t_0 \in (0, \tmax)$ there is $C \gt 0$ such that
  \begin{align}\label{eq:ws_est:w_sim_s}
    \frac sC \le w(s, t) \le Cs \qquad \text{for $s \in [0, R^n]$ and $t \in [0, t_0]$}.
  \end{align}
\end{lemma}
\begin{proof}
  For fixed $t \in [0, \tmax)$ and $s \in [0, R^n]$, the mean value theorem provides us with $\xi \in (0, s)$ such that
  $w(s, t) = s w_s(\xi, t)$,
  which already proves \eqref{eq:ws_est:main} since $w_s$ is decreasing by Lemma~\ref{lm:ur_le_0} and \eqref{eq:def_w:w_s=u}.
  Moreover, a consequence thereof is \eqref{eq:ws_est:w_sim_s},
  since $w_s$ is positive and bounded in $[0, R^n] \times [0, t_0]$ for any $t_0 \in (0, \tmax)$
  by Lemma~\ref{lm:local_ex} and \eqref{eq:def_w:w_s=u}.
\end{proof}

\section{A supersolution to a superlinear ODE: finite-time blow-up}\label{sec:phi}
We will construct initial data leading to finite-time blow-up and hence prove Theorem~\ref{th:main} in this section.
As already mentioned in the introduction, our argument is based on constructing a function $\phi$
which cannot exist globally, implying that the solution of \eqref{prob} also can only exist on a finite time interval.
In fact, we define $\phi$ as in \cite{BlackEtAlRelaxedParameterConditions2020} or \cite{WinklerFinitetimeBlowupLowdimensional2018};
that is, for given $u_0$ as in \eqref{eq:main:reg_u0} and $\gamma \in (0, 2)$, we set
\begin{align}\label{eq:phi:def_phi}
  \phi(s_0, t) \defs \int_0^{s_0} s^{-\gamma} (s_0 - s) w(s, t) \ds, \qquad s_0 \in (0, R^n), t \in (0, \tmax).
\end{align}
However, as the parameter $\gamma$ herein may be larger than $1$
(unlike as in \cite{BlackEtAlRelaxedParameterConditions2020} or \cite{WinklerFinitetimeBlowupLowdimensional2018}),
some care is needed for calculating the time derivative of $\phi$.
This is done in the following

\begin{lemma}\label{lm:phi}
  Suppose that $u_0$ complies with \eqref{eq:main:reg_u0}.
  Let $\gamma \in (0, 2)$ and $\phi$ be as in \eqref{eq:phi:def_phi}.
  For every $s_0 \in (0, R^n)$,
  $\phi(s_0, \cdot)$ belongs to $C^0([0, \tmax)) \cap C^1((0, \tmax))$ and fulfills
  \begin{align}\label{eq:phi:phi_ode}
          \phi_t(s_0, t)
    &\ge  n^2 \int_0^{s_0} s^{2-\frac2n - \gamma} (s_0 - s) w_{ss} \ds \notag \\
    &\pe  + n \int_0^{s_0} s^{-\gamma} (s_0 - s) w w_s \ds \notag \\
    &\pe  - n \ol m(t) \int_0^{s_0} s^{1-\gamma} (s_0 - s) w_s \ds \notag \\
    &\pe  - n^{\kappa-1} \mu \int_0^{s_0} s^{-\gamma} (s_0 - s) \int_0^s w_s^\kappa(\sigma, t) \dsigma \ds \notag \\
    &\sfed I_1(s_0, t) + I_2(s_0, t) + I_3(s_0, t) + I_4(s_0, t)
    \qquad \text{for all $t \in (0, \tmax)$}.
  \end{align}
\end{lemma}
\begin{proof}
  We first fix $s_0 \in (0, R^n)$
  and note that $\phi(s_0, \cdot) \in C^0([0, \tmax))$ because of \eqref{eq:ws_est:w_sim_s} and $1 - \gamma \gt -1$.
  Letting $0 \lt t_0 \lt t_1 \lt \tmax$,
  we then make use of Lemma~\ref{lm:def_w} and Lemma~\ref{lm:ws_est} to obtain $c_1, c_2, c_3, c_4\gt 0$ such that
  \begin{align*}
    w(s, t) \le c_1 s, \quad
    w_s(s, t) \le c_2, \quad
    |w_{ss}(s, t)| \le c_3
    \quad \text{and} \quad
    \ol m(t) \le c_4
    \qquad \text{for $(s, t)\in [0, s_0] \times [t_0, t_1]$}.
  \end{align*}
  Recalling \eqref{eq:def_w:w_pde_exact}, we obtain
  \begin{align*}
    &\pe  \ddt \left( s^{-\gamma} (s_0 - s) w\right) \\
    &=    \left(
            n^2 s^{2-\frac2n} w_{ss}
            + n w w_s
            - n \ol m(t) s w_s
            + \lambda w
            - n^{\kappa-1} \mu \int_0^s w_s^\kappa(\sigma, t) \dsigma
          \right) s^{-\gamma} (s_0 - s)
  \end{align*}
  for $s \in (0, s_0)$ and $t \in (0, \tmax)$,
  so that
  \begin{align*}
          \left| \ddt \left( s^{-\gamma} (s_0 - s) w(s, t) \right) \right|
    &\le  \left(
            n^2 c_3 s_0^{1 - \frac2n}
            + n c_1 c_2 
            + n c_2 c_4
            + \lambda c_1 
            + n^{\kappa-1} \mu c_2^\kappa
          \right) s^{1 - \gamma} (s_0 - s)
  \end{align*}
  for all $s \in (0, s_0)$ and $t \in (t_0, t_1)$.
  Again due to $1 - \gamma \gt -1$, we therefore have $\phi(s_0, \cdot) \in C^1((0, \tmax))$ 
  and
  \begin{align*}
          \phi_t(s_0, t) 
    &=    I_1(s_0, t) + I_2(s_0, t) + I_3(s_0, t)
          + \lambda \int_0^{s_0} s^{-\gamma} (s_0 - s) w \ds
          + I_4(s_0, t)
  \end{align*}
  for all $t \in (0, \tmax)$,
  which due to $\lambda \ge 0$ implies \eqref{eq:phi:phi_ode}.
\end{proof}

Aiming to derive that $\phi$ is a supersolution to a superlinear ODE,
we now estimate the terms $I_1, \dots, I_4$ in \eqref{eq:phi:phi_ode}
and begin with $I_4$, the term stemming from the logistic source.
In the following proof, we will crucially make use of the estimate \eqref{eq:ws_est:main}
to improve on corresponding results obtained by the predecessors~%
\cite{BlackEtAlRelaxedParameterConditions2020} and~\cite{WinklerFinitetimeBlowupLowdimensional2018}.

\begin{lemma}\label{lm:i4}
  Let $I_2$ and $I_4$ be as in \eqref{eq:phi:phi_ode}.

  \begin{enumerate}
    \item[(i)]
      If $\kappa = 2$, $\gamma \gt 1$ and $u_0$ fulfills \eqref{eq:main:reg_u0},
      then
      \begin{align}\label{eq:i4:statement_1}
        I_4(s_0, t) \ge - \frac{\mu}{\gamma - 1} I_2(s_0, t)
        \qquad \text{for all $s_0 \in (0, R^n)$ and $t \in (0, \tmax)$}.
      \end{align}

    \item[(ii)]
      Let $\kappa \in (1, 2)$ and $\gamma \in (\frac{2(\kappa-1)}{\kappa}, 1)$.
      There exists $C_4 \gt 0$ such that whenever $u_0$ fulfills \eqref{eq:main:reg_u0}, then
      \begin{align}\label{eq:i4:statement_2}
        I_4(s_0, t) \ge C_4 s_0^{\frac{2-\kappa}{2}} I_2^\frac{\kappa}{2}(s_0, t)
        \qquad \text{for all $s_0 \in (0, \min\{1, R^n\})$ and $t \in (0, \tmax)$}.
      \end{align}
  \end{enumerate}
\end{lemma}
\begin{proof}
  We let $\gamma \in (0, \infty) \setminus \{1\}$
  and also fix $u_0$ as in \eqref{eq:main:reg_u0} but will make sure that $C_4$ can be taken independently of $u_0$.
  By Fubini's theorem, we first observe that
  \begin{align}\label{eq:i4:first_est}
          - I_4(s_0, t)
    &=    - n^{\kappa-1} \mu \int_0^{s_0} s^{-\gamma} (s_0 - s) \int_0^s w_s^\kappa(\sigma, t) \dsigma \ds \notag \\
    &=    - n^{\kappa-1} \mu \int_0^{s_0} \left( \int_\sigma^{s_0} s^{-\gamma} (s_0 - s) \ds \right) w_s^{\kappa}(\sigma, t) \dsigma \notag \\
    &\ge  - n^{\kappa-1} \mu \int_0^{s_0} \left( \int_\sigma^{s_0} s^{-\gamma} \ds \right) (s_0 - \sigma)  w_s^{\kappa}(\sigma, t) \dsigma \notag \\
    &=    - \frac{n^{\kappa-1} \mu}{1-\gamma} \int_0^{s_0} \left( s_0^{1-\gamma} - s^{1-\gamma} \right) (s_0 - s)  w_s^{\kappa}(s, t) \ds
  \end{align}
  for all $s_0 \in (0, R^n)$ and $t \in (0, \tmax)$.

  In the case of $\gamma \gt 1$ and $\kappa = 2$, we drop a positive term and employ \eqref{eq:ws_est:main} in calculating
  \begin{align*}
          - I_4(s_0, t)
    &\ge  - \frac{n \mu}{\gamma-1} \int_0^{s_0} s^{1-\gamma} (s_0 - s) w_s^2 \ds \\
    &\ge  - \frac{n \mu}{\gamma-1} \int_0^{s_0} s^{-\gamma} (s_0 - s) w w_s \ds \\
    &=    - \frac{\mu}{\gamma-1} I_2(s_0, t) 
    \qquad \text{for all $s_0 \in (0, R^n)$ and $t \in (0, \tmax)$},
  \end{align*}
  which already implies \eqref{eq:i4:statement_1}.
  
  If on the other hand $\gamma \in (0, 1)$ and $\kappa \in (0, 1)$,
  going back to \eqref{eq:i4:first_est} and making use of use of \eqref{eq:ws_est:main},
  we see that
  \begin{align}\label{i4:part2_1}
          I_4(s_0, t)
    &\ge  - \frac{n^{\kappa-1} \mu }{1-\gamma} s_0^{1-\gamma} \int_0^{s_0} (s_0 - s) w_s^{\kappa}(s, t) \ds
     \ge  - \frac{n^{\kappa-1} \mu }{1-\gamma} R^{n(1-\gamma)} \int_0^{s_0} s^{-\frac{\kappa}{2}} (s_0 - s) (w w_s)^\frac{\kappa}{2} \ds
  \end{align}
  for all $s_0 \in (0, R^n)$ and $t \in (0, \tmax)$.
  By Hölder's inequality (with exponents $\frac{2}{2-\kappa}, \frac2\kappa$), we have therein
  \begin{align}\label{i4:part2_2}
          \int_0^{s_0} s^{-\frac{\kappa}{2}} (s_0 - s) (w w_s)^\frac{\kappa}{2} \ds
    &=    \int_0^{s_0} s^{-\frac{(1-\gamma) \kappa}{2}} (s_0 - s) (s^{-\gamma} w w_s)^\frac{\kappa}{2} \ds \notag \\
    &\le  \left( \int_0^{s_0} s^{-\frac{(1-\gamma)\kappa}{2-\kappa}} (s_0 - s) \ds \right)^\frac{2-\kappa}{2}
          \left( \intom s^{-\gamma} (s_0 - s) w w_s \ds \right)^\frac{\kappa}{2}
  \end{align}
  for all $s_0 \in (0, R^n)$ and $t \in (0, \tmax)$.
  We assume now moreover that $\gamma \gt \frac{2(\kappa-1)}{\kappa}$
  and hence $\gamma-1 \gt \frac{\kappa-2}{\kappa}$ as well as $a \defs \frac{(\gamma-1)\kappa}{2-\kappa} \gt -1$,
  so that applying Lemma~\ref{lm:beta} (with $B$ as in that lemma) gives
  \begin{align}\label{i4:part2_3}
          \int_0^{s_0} s^{-\frac{(1-\gamma)\kappa}{2-\kappa}} (s_0 - s) \ds
    =     B s_0^{a+2}
    \le   B s_0
    \qquad \text{for all $s_0 \in (0, \min\{1, R^n\})$}.
  \end{align}
  Finally, combining \eqref{i4:part2_1}--\eqref{i4:part2_3} and the definition of $I_2$ yields \eqref{eq:i4:statement_2}
  for some $C_4 \gt 0$ independent of $u_0$.
\end{proof}
 
The remaining integrals in \eqref{eq:phi:phi_ode} can be estimated
as in \cite{WinklerFinitetimeBlowupLowdimensional2018} or \cite{BlackEtAlRelaxedParameterConditions2020}.
However, at least for the statement concerning $I_1$,
we would like to give a full proof here in order to show the basis of the restriction on $\kappa$ in Theorem~\ref{th:main}.
Indeed, while in Lemma~\ref{lm:i4} above, $\gamma$ has to be taken sufficiently large,
for estimating $I_1$, we need $\gamma$ to be suitably small.
We will obtain finite-time blow-up precisely in the cases where the set of admissible $\gamma$ for both these lemmata is nonempty.
Moreover, compared to \cite{WinklerFinitetimeBlowupLowdimensional2018},
the proof below makes use of the estimate \eqref{eq:ws_est:main} and is hence somewhat shorter.
\begin{lemma}\label{lm:i1}
  Let $\gamma \in (0, 2 - \frac4n)$.
  There is $C_1 \gt 0$ such that
  whenever $u_0$ satisfies \eqref{eq:main:reg_u0} and $I_1, I_2$ are as in \eqref{eq:phi:phi_ode},
  then
  \begin{align*}
        I_1(s_0, t)
    \ge - C_1 s_0^{\frac{3-\gamma}{2} - \frac2n} I_2^\frac12(s_0, t)
    \qquad \text{for all $s_0 \in (0, R^n)$ and $t \in (0, \tmax)$}.
  \end{align*}
\end{lemma}
\begin{proof}
  For convenience, we fix $u_0$ as in \eqref{eq:main:reg_u0},
  albeit we emphasize that the constants below do not depend on $u_0$.
  An integration by parts gives
  \begin{align*}
     &\pe \int_0^{s_0} s^{2-\frac2n - \gamma} (s_0 - s) w_{ss} \ds \\
     &=   - \left(2 - \frac2n - \gamma\right) \int_0^{s_0} s^{1 - \frac2n - \gamma} (s_0 - s) w_s \ds
          + \int_0^{s_0} s^{2-\frac2n - \gamma} w_s \ds
          + \left[ s^{2-\frac2n - \gamma} (s_0 - s) w_s \right]_0^{s_0}
    \qquad \text{in $(0, \tmax)$}.
  \end{align*}
  Herein, the second term on the right hand side is positive
  and the last one is zero because of $\gamma \lt 2 - \frac4n \lt 2 - \frac2n$.

  Setting $c_1 \defs 2 - \frac2n - \gamma \gt 0$,
  we hence infer from \eqref{eq:ws_est:main} and Hölder's inequality that
  \begin{align*}
    &\pe  \int_0^{s_0} s^{2 - \frac2n - \gamma} (s_0 - s) w_{ss} \ds \\
    &\ge  - \left(2 - \frac2n - \gamma\right) \int_0^{s_0} s^{1 - \frac2n - \gamma} (s_0 - s) w_s \ds \\
    &\ge  - c_1 \int_0^{s_0} s^{\frac12 - \frac2n - \gamma} (s_0 - s) (w w_s)^\frac12 \ds \\
    &\ge  - c_1 \left( \int_0^{s_0} s^{1 - \gamma - \frac4n} (s_0 - s) \ds \right)^\frac12
            \left( \int_0^{s_0} s^{-\gamma} (s_0 - s) w w_s \ds \right)^\frac12
  \qquad \text{holds in $(0, \tmax)$}.
  \end{align*}
  Since $\gamma \lt 2 - \frac4n$ and hence $a \defs 1 - \gamma - \frac4n \gt -1$,
  Lemma~\ref{lm:beta} asserts that (with $B$ as in that lemma)
  \begin{align*}
      \left( \int_0^{s_0} s^{1 - \gamma - \frac4n} (s_0 - s) \ds \right)^\frac12
    = B^\frac12 s_0^{\frac{3-\gamma}{2} - \frac2n}
    \qquad \text{for all $s_0 \in (0, R^n)$},
  \end{align*}
  so that the statement follows by the definitions of $I_1$ and $I_2$.
\end{proof}

Next, for estimating the integrals $I_2$ and $I_3$ in \eqref{eq:phi:phi_ode},
we basically recall the corresponding results from \cite{WinklerFinitetimeBlowupLowdimensional2018}.
\begin{lemma}\label{lm:i2_i3}
  There exist $C_2, C_3 \gt 0$ such that for $u_0$ satisfying \eqref{eq:main:reg_u0}, we have
  \begin{alignat}{2}\label{eq:i2_i3:i2}
    I_2(s_0, t) &\ge C_2 s_0^{-(3-\gamma)} \phi^2(s_0, t)
    &&\qquad \text{for all $s_0 \in (0, R^n)$ and $t \in (0, \tmax)$}
  \intertext{and}\label{eq:i2_i3:i3}
    I_3(s_0, t) &\ge - C_3 \left(\intom u_0\right) s_0^{\frac{3-\gamma}{2}} I_2^\frac12(s_0, t)
    &&\qquad \text{for all $s_0 \in (0, R^n)$ and $t \in (0, \min\{1, \tmax\})$},
  \end{alignat}
  where $\phi$ is in \eqref{eq:phi:def_phi} and $I_2, I_3$ are defined in \eqref{eq:phi:phi_ode}.
\end{lemma}
\begin{proof}
  Arguing as in \cite[Lemma~4.4]{WinklerFinitetimeBlowupLowdimensional2018}, we obtain
  \begin{align}\label{eq:i2_i3:calc}
          \phi(s_0, t)
    &\le  c_1 s_0^{\frac{3-\gamma}{2}}
          I_2^\frac12(s_0, t)
    \qquad \text{for all $s_0 \in (0, R^n)$ and $t \in (0, \tmax)$}
  \end{align}
  for some $c_1 \gt 0$ independent of $u_0$.
  Taking both the left and the right hand side therein to the power $2$ already yields \eqref{eq:i2_i3:i2}.
  Moreover, as
  \begin{align*}
        I_3(s_0, t)
    =   - n \ol m(t) \int_0^{s_0} s^{1-\gamma} (s_0 - s) w_s(s, t) \ds
    \ge - \frac{n}{|\Omega|} \left( \intom u_0 \right) \ure^{\lambda t} \phi(s_0, t)
    \qquad \text{for $(s_0, t) \in (0, R^n) \times (0, \tmax)$}
  \end{align*}
  by Lemma~\ref{lm:u_bdd_l1} and \eqref{eq:ws_est:main},
  another consequence of \eqref{eq:i2_i3:calc} is \eqref{eq:i2_i3:i3}.
\end{proof}

As a final preparation, we note that, under certain circumstances, $\phi(s_0, 0)$ can be shown to be sufficiently large.
\begin{lemma}\label{lm:phi_0}
  For every $m_1 \gt 0$, there exists $C_0 \gt 0$ with the following property:
  Let $s_0 \in (0, R^n)$, set $s_1 \defs \frac{s_0}{4}$ as well as $r_1 \defs s_1^\frac1n$
  and suppose that $u_0$ fulfills \eqref{eq:main:reg_u0} as well as $\int_{B_{r_1}(0)} u_0 \ge m_1$.
  Then $\phi(s_0, 0) \ge C_0 s_0^{2-\gamma}$.
\end{lemma}
\begin{proof}
  See \cite[estimate~(5.5)]{WinklerFinitetimeBlowupLowdimensional2018};
  the main idea is to use the monotonicity of $w_0$ which in turn is implied by nonnegativity of $u_0$.
\end{proof}

A combination of the results obtained above now reveals that for initial data whose mass is sufficiently concentrated near the origin,
the corresponding solution cannot exist globally in time.
Again, the argument is not too different from \cite{WinklerBlowupHigherdimensionalChemotaxis2011} or \cite{WinklerFinitetimeBlowupLowdimensional2018},
but we choose to give it nonetheless in order to show that $s_0$ and $u_0$ can be chosen in such a way
that $\phi$ would blow up in finite time if $(u, v)$ were a global solution.
\begin{lemma}\label{lm:tmax_le_12}
  Let $m_0 \gt m_1 \gt 0$ and suppose that \eqref{eq:main:cond} holds.
  There exists $r_1 \in (0, R)$ such that whenever $u_0$ fulfills \eqref{eq:main:reg_u0} and \eqref{eq:main:mass_u0},
  then $\tmax \le \frac12$.
\end{lemma}
\begin{proof}
  Let us begin by fixing some parameters.
  If \eqref{eq:main:cond:1} holds,
  then $\kappa \in (1, \frac n2)$ and hence
  \begin{align*}
    \frac{2(\kappa-1)}{\kappa} - \left(2-\frac4n\right) \lt \frac{2 \cdot \frac{n-2}{2}}{\frac n2} - \frac{2(n-2)}{n} = 0.
  \end{align*}
  As additionally $\kappa \lt 2$,
  we may hence choose $\gamma \in (\frac{2(\kappa-1)}{\kappa}, \min\{2-\frac4n, 1\})$.
  We moreover fix an arbitrary $\eps \gt 0$ and apply Lemma~\ref{lm:i4}~(ii)
  as well as Young's inequality (with exponents $\frac{2}{2-\kappa}$, $\frac{2}{\kappa}$)
  to obtain $C_4' \gt 0$ with
  \begin{align}\label{eq:tmax_le_12:i4}
        I_4(s_0, t) \ge - \frac{\mu}{\mu+\eps} I_2(s_0, t) - C_4' s_0
        \qquad \text{for all $s_0 \in (0, \min\{1, R^n\})$ and $t \in (0, \tmax)$},
  \end{align}
  whenever $u_0$ satisfies \eqref{eq:main:reg_u0}
  and where $I_2$ and $I_4$ are as in \eqref{eq:phi:phi_ode}.

  Suppose now that on the other hand \eqref{eq:main:cond:2} holds.
  Because of $\mu \in (0, \frac{n-4}{n})$, we may then choose $\gamma \in (1 + \mu, 2 - \frac4n)$.
  Setting moreover $\eps \defs \gamma - 1 - \mu \gt 0$, an application of Lemma~\ref{lm:i4}~(i) reveals that
  \eqref{eq:tmax_le_12:i4} holds also in this case (with $C_4' \defs 0$ and for all $u_0$ complying with \eqref{eq:main:reg_u0}).

  In both cases, the definition of $\gamma$ entails $0 \lt \gamma \lt 2 - \frac4n$,
  hence by Lemma~\ref{lm:phi}, Lemma~\ref{lm:i1}, \eqref{eq:i2_i3:i3}, \eqref{eq:tmax_le_12:i4}, Young's inequality and \eqref{eq:i2_i3:i2},
  there are $c_1, c_2 \gt 0$ such that
  \begin{align}\label{eq:main_proof:phi_t}
          \phi_t(s_0, t)
    &\ge  I_1(s_0, t) + I_2(s_0, t) + I_3(s_0, t) + I_4(s_0, t) \notag \\
    &\ge  \frac{\eps}{\mu + \eps} I_2(s_0, t)
          - \left(
            C_1 s_0^{\frac{3-\gamma}{2} - \frac2n}
            + C_3 m_0 s_0^{\frac{3-\gamma}{2}}
            \right) I_2^\frac12(s_0, t) 
          - C_4' s_0 \notag \\
    &\ge  c_1 I_2(s_0, t)
          - c_2 s_0^{\min\{3-\gamma - \frac4n, 3-\gamma, 1\}} \notag \\
    &\ge  C_2 c_1 s_0^{-(3-\gamma)} \phi^2(s_0, t)
          - c_2 s_0
  \end{align}
  for all $s_0 \in (0, \min\{1, R^n\})$, $t \in (0, \min\{1, \tmax\})$ and $u_0$ satisfying \eqref{eq:main:reg_u0} as well as $\intom u_0 = m_0$,
  where $\phi, I_1, \dots, I_4$ are as in Lemma~\ref{lm:phi}, $C_1$ is as in Lemma~\ref{lm:i1} and $C_2, C_3$ are as in Lemma~\ref{lm:i2_i3}.
  
  For $s_0 \gt 0$, we set $c_3 \defs C_2 c_1$,
  \begin{align*}
    d_1(s_0) \defs c_3 s_0^{-(3-\gamma)}, \quad
    d_2(s_0) \defs c_2 s_0, \quad
    d_3(s_0) \defs \left( \frac{d_2(s_0)}{d_1(s_0)} \right)^\frac12 \quad \text{and} \quad
    \phi_0(s_0) \defs C_0 s_0^{2-\gamma},
  \end{align*}
  where $C_0$ is as in Lemma~\ref{lm:phi_0}.
  We observe that $d_1(s_0) \ra \infty$ for $s_0 \sea 0$ since $3-\gamma \gt 1 \gt 0$.
  Therefore, noting further that
  \begin{align*}
          \frac12 ( 1 + 3 - \gamma )
    &=    2 - \frac{\gamma}{2}
    \gt   2 - \gamma,
  \end{align*}
  we may also fix $s_0 \in (0, \min\{1, R^n\})$ so small that
  \begin{align}\label{eq:main_proof:phi0_large}
        \phi_0(s_0)
    \ge d_3(s_0)
        + \frac{2}{d_1(s_0)}.
  \end{align}
  Moreover, we now fix $u_0$ not only complying with \eqref{eq:main:reg_u0} but also with \eqref{eq:main:mass_u0} for $r_1 \defs (\frac{s_0}{4})^\frac1n$
  and will show that the corresponding solution given by Lemma~\ref{lm:local_ex} blows up in finite time.
  From \eqref{eq:main_proof:phi_t} and Lemma~\ref{lm:phi_0}, we infer that $\phi(s_0, \cdot)$ satisfies
  \begin{align}\label{eq:main_proof:phi_ode}
    \begin{cases}
      \phi_t(s_0, t) \ge d_1(s_0) \phi^2(s_0, t) - d_2(s_0) & \text{for all $t \in (0, \min\{1, \tmax\})$}, \\
      \phi(s_0, 0) \ge \phi_0(s_0).
    \end{cases}
  \end{align}

  Since \eqref{eq:main_proof:phi0_large} implies $\phi_0(s_0) \ge d_3(s_0)$
  and because of $d_1(s_0) d_3(s_0)^2 - d_2(s_0) = 0$,
  the comparison principle and \eqref{eq:main_proof:phi_ode} assert $\phi(s_0, t) \ge d_3(s_0)$ for all $t \in (0, \min\{1, \tmax\})$,
  so that by \eqref{eq:main_proof:phi_ode} we have
  \begin{align*}
          \phi_t(s_0, t)
    &\ge  d_1(s_0) \left( \phi^2(s_0, t) - d_3(s_0)^2 \right) \\
    &\ge  d_1(s_0) \left( \phi(s_0, t) - d_3(s_0) \right)^2
    \qquad \text{for all $t \in (0, \min\{1, \tmax\})$}.
  \end{align*}
  Dividing by the right hand side therein yields upon an integration in time
  \begin{align*}
          t
    &=    \int_0^t 1 \ds
     \le  \int_{\phi(s_0, 0)}^{\phi(s_0, t)} \frac{\dsigma}{d_1(s_0) (\sigma - d_3(s_0))^2}
     \le  \left[ -\frac{1}{d_1(s_0) (\sigma - d_3(s_0))} \right]_{\phi_0(s_0)}^{\infty}
     \le  \frac12
  \end{align*}
  for all $t \in (0, \min\{1, \tmax\})$, implying $\tmax \le \frac12$.
\end{proof}

Finally, we conclude that Theorem~\ref{th:main} is now merely a direct consequence of the lemmata above.
\begin{proof}[Proof of Theorem~\ref{th:main}]
  Lemma~\ref{lm:tmax_le_12} asserts that there is $r_1 \in (0, R)$ such that under the conditions of Theorem~\ref{th:main},
  the maximal existence time $\tmax$ is finite.
  By Lemma~\ref{lm:ur_le_0} and Lemma~\ref{lm:local_ex},
  this then implies
  $u(0, t) = \|u(\cdot, t)\|_{\leb\infty} \ra \infty$ as $t \nea \tmax$.
\end{proof}

\section*{Acknowledgments}
\small The author is partly supported by the German Academic Scholarship Foundation
and by the Deutsche Forschungsgemeinschaft within the project \emph{Emergence of structures and advantages in
cross-diffusion systems}, project number 411007140.

\end{document}